\documentclass{hmjart}
\usepackage{amssymb}
\usepackage{graphicx}
\newtheorem{thm}{Theorem}[section]
\newtheorem{cor}[thm]{Corollary}
\newtheorem{lem}[thm]{Lemma}
\newtheorem{prop}[thm]{Proposition}
\theoremstyle{definition}

\theoremstyle{remark}

\theoremstyle{conjecture}

\numberwithin{equation}{section}

\begin{document}
\hmjlogo{}{}{}{}                
\title[Almost universality of a sum of norms]{Almost universality of a sum of norms}


\author[Park, Jeongho]{Jeongho Park}
\address{Department of Mathematics\\
  POSTECH \\
  San 31 Hyoja Dong, Nam-Gu, Pohang 790-784, KOREA.}

\email{pkskng@postech.ac.kr}
\grants{Supported by the National Research Foundation of Korea (NRF) grant funded by the Korea government(MEST)
(2010-0026473).}

\begin{abstract}
In this paper the author considers a particular type of polynomials with integer coefficients, consisting of a perfect power and two norm forms of abelian number fields with coprime discriminants. It is shown that such a polynomial represents every natural number with only finitely many exceptions. The circle method is used, and the local class field theory played a central role in estimating the singular series.
\end{abstract}

\maketitle

\section{Introduction}
Let $f = f(\overrightarrow{\zeta}) = f(\zeta_1,\zeta_2,\cdots)$ be an integer-valued polynomial in many variables that is \emph{locally universal}, meaning that for any $n\in\mathbb{N}$ and a prime $p$
\begin{equation}\label{eqn_f zeta equals n}
    f(\overrightarrow{\zeta}) = n
\end{equation}
is soluble with $\zeta_i \in \mathbb{Z}_p$. When \eqref{eqn_f zeta equals n} is soluble for all $n\in\mathbb{N}$ in integers $\zeta_1,\zeta_2,\cdots$, $f$ is said to be \emph{universal}, and when it is soluble for all but finitely many $n\in\mathbb{N}$ it is \emph{almost universal}. It is quite often the case that $f$ is not universal or even almost universal as indicated by the Brauer-Manin obstruction. The failure of the Hasse principle, though, can be in some sense overcome if we allow more variables in the equation. It is a general phenomenon that if a surface $\mathcal{S}(n)$ of a higher dimension, defined by $F(\overrightarrow{z})=n$, contains the original surface $\mathcal{S}_0(n)$ given by $f(\overrightarrow{\zeta})=n$ as a section, then it becomes almost universal, provided their codimension is sufficiently large. The Hardy-Littlewood method provides a powerful tool for estimating the number of variables necessary for this.

Although this analytic method is powerful for many problems of additive nature, it has some technical limitations and a factor of $\log (\deg f)$ is invincible in general; most of the results satisfy $\dim(\mathcal{S}(n)) / \dim(\mathcal{S}_0(n)) \gg \log (\deg f)$. The term $\log (\deg f)$ arises from the technical difficulty to handle the minor arcs. If one considers a specific polynomial that is irrelevant of this difficulty, it is provable that $\log (\deg f)$ reduces to $O(1)$.

In the 1960s, Birch, Davenport and Lewis already observed this and considered a specific type of $k$-forms. Let $K$, $E$ be the number fields of degree $k$ with integral bases $\{\omega_1,\cdots,\omega_k\}$, $\{\omega_1',\cdots,\omega_k'\}$. Denote the norm of $x_1 \omega_1 + \cdots + x_k \omega_k$ and $y_1 \omega_1' + \cdots + y_k \omega_k'$by $N_K(x_1,\cdots,x_k)$ and $N_E(y_1,\cdots,y_k)$. In ~\cite{Birch} they considered $z^k + N_K(x_1,\cdots,x_k) + N_E(y_1,\cdots,y_k)$ under some provisos, reached the following theorem as their main result.

\begin{thm}\label{thm_Birch Davenport Lewis}
Suppose $K$, $E$ are not both totally complex, and let $\delta_K$, $\delta_E$ denote their discriminants. Suppose that, for every prime $p$ dividing both $\delta_K$ and $\delta_E$, the equation
\begin{equation}\label{eqn_Birch Davenport Lewis}
    n = z^k + N_K(x_1,\cdots,x_k) + N_E(y_1,\cdots,y_k)
\end{equation}
has a non-sinigular solution in the $p$-adic field. Then \eqref{eqn_Birch Davenport Lewis} has infinitely many solutions in integers.
\end{thm}

The strategy in their work was to follow the routine procedure of the circle method, except for the treatment of the singular series. A typical way of dealing with the singular series is to obtain an estimate of exponential sums     $S_{a,q} = \sum_{\overrightarrow{\zeta}\;\text{mod } q} e(a f(\overrightarrow{\zeta})/q)$ by factoring them into products of exponential sums over a single variable. But a norm form $N_K(x_1,\cdots,x_k)$ has many off-diagonal terms and it is not easy to do that directly. In ~\cite{Birch} instead, a technique invented by Birch was used to get an auxiliary lemma of the type $\sum_{q\leq X} q^{-2k} \sum_{\substack{a=1\\(a,q)=1}}^q |S_{a,q}|^2 \ll X^{\epsilon}$. They also showed the positivity of the singular series, but without giving an estimate of its range.

\medskip
In this study we prove the almost universality of the polynomial $z^{k_0} + N_K(\overrightarrow{x}) + N_E(\overrightarrow{y})$ in a different setting that enlightens several new aspects. We treat abelian number fields $K$, $E$ whose degrees $k_1$, $k_2$ may be different and can be totally imaginary as well. We also provide an effective bound for the singular series. In the last section we will include an example which shows the optimality of our conclusion, i.e., that the sum of two norms can fail to be almost universal even when it is locally universal. There are very few optimal results in the theory of additive problems, but this result seems to provide one of them.

The techniques in ~\cite{Birch} may help us deduce some parts of our results, but will not exhaust the contents of this research. On a technical point of view, the significance of this study is that estimates of the exponential sum $S_{a,q}$ uses local class field theory and a restricted sum, which plays an essential role in dealing with the singular series by removing solutions that are $p$-adically singular for a `bad' prime $p$. (Here `bad' means that either $p$ is too small or it is ramified in one of $K$, $E$.)

This article is organized as follows. In Section ~\ref{sec_Notations and settings} we give some definitions, in particular of the congruence residue classes $\breve{R}_Q$ and $\breve{\frak{B}}(n)$ that will be important in the treatment of the major arcs. In Section ~\ref{sec_major arcs} we reduce the integration over major arcs into singular series and singular integral. The estimation of the singular series follows in three subsequent sections. Section ~\ref{sec_singular series I} covers the exponential sums over bad primes. Section~\ref{sec_algebraic preparation for the singular series} gives some basics about algebraic number theory, which will be used in Section~\ref{sec_singular series II} to design a particular system of representatives of $O_K/{p^L O_K}$, together with local class field theory, in order to obtain a successful bound of $S_{a,p^L}$ for good primes $p$. Section~\ref{sec_singular integral} lays down a fairly classical argument on the singular integral. We omit estimation on the minor arcs because the reasoning in ~\cite{Birch} carries over verbatim. The last section gives the desired asymptotic formula and describes how to construct a sum of two norms that is locally universal but is not almost universal, demonstrating why the term $z^k$ is necessary in our problem.

\section{Notations and Settings}\label{sec_Notations and settings}

Let $K$ and $E$ be abelian number fields with ring of integers $O_K$, $O_E$ and integral bases $\{ \varphi_1, \cdots, \varphi_{k_1}  \}$ and $\{ \psi_1, \cdots, \psi_{k_2} \}$ whose discriminants $\delta_K$, $\delta_E$ are relatively prime. Since $(\delta_K, \delta_E)=1$, for each rational prime $p$ at least one of the inertial groups $\frak{T}_K(p)$, $\frak{T}_E(p)$ is trivial. Suppose $\frak{T}_K(p)$ is trivial. Let $\frak{p}_1, \cdots, \frak{p}_g$ be the prime ideals in $K$ above $p$, and $O_{K,\frak{p}_i}$ the completion of $O_K$ at $\frak{p}_i$. The local class field theory says that the local norm map $N_{K,i} : \left(O_{K,\frak{p}_i}\right)^{\ast} \rightarrow \left(\mathbb{Z}_p\right)^{\ast}$ is surjective. We also know that $N_K(\alpha) = \prod_{i=1}^g N_{K,i}(\alpha)$ for $\alpha \in K$, where we consider $\alpha$, $N_K(\alpha)$ as elements in adequate completions of $K$ \cite{Kato}. Since $O_K$ embeds into $\prod_{i=1}^g O_{K,\frak{p}_i}$ as a dense subset, combining the surjectivity of $N_{K,i}$, the local universality of $N_K(\overrightarrow{x}) + N_E(\overrightarrow{y}) = N_{K|\mathbb{Q}}(\overrightarrow{x}\cdot\overrightarrow{\varphi}) + N_{E|\mathbb{Q}}(\overrightarrow{y}\cdot\overrightarrow{\psi})$ will follow if the local norm maps $N_{K,1},\cdots,N_{K,g}$ are all continuous. But $N_{K,i}$ is indeed continuous as we easily see. Choose a uniformizer $\pi \in \frak{p}_i$ in $O_{K,\frak{p}_i}$. Any element $\sigma \in Gal(K_{\frak{p}_i} | \mathbb{Q}_p)$ maps $\pi$ to another uniformizer in $O_{K,\frak{p}_i}$, so for an element $y \in O_{K,\frak{p}_i}$ we can write $(y\pi^{\ell})^{\sigma} = \tilde{y}(\sigma) \pi^{\ell}$ for some $\tilde{y}(\sigma) \in O_{K,\frak{p}_i}$. Thus for $x,y \in O_{K,\frak{p}_i}$, one has
\begin{equation*}
    N_{K,i} \left(x + y \pi^{\ell}\right) = \prod_{\sigma}\left(x + y \pi^{\ell}\right)^{\sigma} = \prod_{\sigma}\left(x^{\sigma} + \tilde{y}(\sigma) \pi^{\ell}\right) = N_{K,i}(x) + \pi^{\ell}z
\end{equation*}
for some $z \in O_{K,\frak{p}_i}$. But $\pi^{\ell} z = N_{K,i}\left(x + y \pi^{\ell}\right) - N_{K,i}(x) \in \frak{p}_i^{\ell} \cap \mathbb{Z}_p$, whence we have $N_{K,i}\left(x + y \pi^{\ell}\right) \equiv N_{K,i}(x)$ mod $p^{\ell}$. It follows that $N_{K,i}$ is a continuous map, and the sum of norms $N_K(\overrightarrow{x}) + N_E(\overrightarrow{y})$ is therefore locally universal.

Throughout this paper we write
\begin{multline*}
     F(\overrightarrow{z}) = F(z_0,z_1,\cdots,z_{k_1+k_2}) = F(z_0, x_1,\cdots,x_{k_1}, y_{1},\cdots,y_{k_2}) \\ = F(z_0, \overrightarrow{x}, \overrightarrow{y}) = z_0^{k_0} + N_K(\overrightarrow{x}) + N_E(\overrightarrow{y}).
\end{multline*}

$e(\alpha)$ denotes $e^{2 \pi i \alpha}$ and $v_p$ denotes the $p$-adic valuation as in tradition. $\mathbb{P}$ denotes the set of rational primes. We put $Q = \prod_{p \in \breve{\mathbb{P}}}p$ where
 \begin{equation*}
 \breve{\mathbb{P}} = \{\; p \in \mathbb{P} \; : \; p < (k_1 k_2)^{2 k_0} \text{ or $p | \delta_K \delta_E$}  \;\}.
 \end{equation*}
In particular $1 \ll Q \ll 1$ for fixed $k_0$, $K$ and $E$.

Let $\breve{R}_q^K$ be a system of representatives $\overrightarrow{x}$ modulo $q$ such that for any prime $p$ in $\breve{\mathbb{P}}$ that divides $q$, one has $N_K(\overrightarrow{x}) \not\equiv 0$ (mod $p$). $\breve{R}_q^E$ is defined in the same manner, and let $\breve{R}_q$ be a system of representatives $\overrightarrow{z} = (z_0,\overrightarrow{x},\overrightarrow{y})$ (mod $q$) with $\overrightarrow{x} \in \breve{R}_q^K$, $\overrightarrow{y} \in \breve{R}_q^E$. Note that $|\breve{R}_{qr}| = |\breve{R}_q| |\breve{R}_r|$ if $(q,r) = 1$. For brevity in the sequel we write $\overrightarrow{z} \equiv \breve{R}_q$ for an integral vector $\overrightarrow{z}$ if $\overrightarrow{z} \equiv \overrightarrow{a}$ (mod $q$) for some $\overrightarrow{a} \in \breve{R}_q$. Observe that $\overrightarrow{z} \equiv \breve{R}_q$ for all $q \geq 1$ if and only if $\overrightarrow{z} \equiv \breve{R}_Q$.

Throughout this article we let $k = \max\{k_0,k_1,k_2\}$ and $X_i = n^{1/{k_i}}$. In the definition of major and minor arcs, we consider the Farey dissection of order $n^{1-\nu}$ where $\nu$ is a fixed small positive number, in particular, less than $\frac{1}{5k}$. For a technical reason in the estimation of the singular integral, we need to choose small boxes $\frak{B}_i$ in $[0,1]$, $[0,1]^{k_1}$ and $[0,1]^{k_2}$ so that $\frak{B} = \frak{B}_0 \times \frak{B}_1 \times \frak{B}_2$ is around a nonsingular point of the Euclidean surface defined by $F(\overrightarrow{\phi}) = 1$. Let $X_i \frak{B}_i = \{ \overrightarrow{r}\; : \; \frac{1}{X_i} \overrightarrow{r} \in \frak{B}_i \}$. The \emph{restricted sum} $\sum_{\overrightarrow{z} \in \breve{\frak{B}}(n)} e(\alpha F(\overrightarrow{z}))$ is defined over
\begin{equation*}
    \breve{\frak{B}}(n) = \{ \overrightarrow{r} \in \left( X_0\frak{B}_0 \times X_1\frak{B}_1 \times X_2\frak{B}_2 \right) \bigcap \mathbb{Z}^{1+k_1+k_2}  \; : \;  \overrightarrow{r} \equiv \breve{R}_Q \}.
\end{equation*}
$\breve{\frak{B}}^K(n)$ is defined similarly using $X_1\frak{B}_1 \bigcap \mathbb{Z}^{k_1}$ and $\breve{R}_Q^K$, and so is $\breve{\frak{B}}^E(n)$.  We want to estimate the number of solutions to $F(\overrightarrow{z}) = n$ with $\overrightarrow{z} \in \breve{\frak{B}}(n)$, viz,
\begin{equation*}
\breve{r}(n) = \int_{c}^{1+c} \sum_{\overrightarrow{z} \in \breve{\frak{B}}(n)} e(\alpha F(\overrightarrow{z}))e(-n \alpha )d\alpha
\end{equation*}
for some real number $c$.

For simplicity we write $\sum_{a\text{ mod* } q}$ to denote the summation over $a$ that runs through $1$ to $q$ under the condition $(a,q)=1$. We define the major arcs for $1 \leq a \leq q \leq n^{\nu}$ and $(a,q)=1$ by
\begin{equation*}
    \frak{M}(q,a) = \{ \alpha\; : \; |\alpha - a/q | \leq n^{-1+\nu} \}, \quad
    \frak{M} = \bigcup_{q \leq n^{\nu}}\bigcup_{a\text{ mod* }q}\frak{M}(q,a)
\end{equation*}
and put the minor arcs $\frak{m} = (n^{\nu-1},1+n^{\nu-1}] - \frak{M}$. Note that $\frak{M}$ is a disjoint union of the segments $\frak{M}(q,a)$.

\section{The Major Arcs}\label{sec_major arcs}
For each rational prime $p$ and $(a,p) = 1$, let
\begin{equation*}
S_{a,p^L}^K = \sum_{\overrightarrow{x} \text{ mod $p^L$}} e\left( \frac{a}{p^L} N_K(\overrightarrow{x}) \right), \quad  \breve{S}_{a,p^L}^K = \sum_{\overrightarrow{x} \in \breve{R}_{p^L}^K} e\left( \frac{a}{p^L} N_K(\overrightarrow{x}) \right).
\end{equation*}
$S_{a,p^L}^E$ and $\breve{S}_{a,p^L}^E$ are defined in the same manner.
For $(a,q)=1$ we also put
\begin{gather*}
\breve{S}_{a,q} = \sum_{\overrightarrow{z} \in \breve{R}_q} e \left( \frac{a}{q}F(\overrightarrow{z}) \right), \quad I(\beta) = \int_{X_0\frak{B}_0 \times X_1\frak{B}_1 \times X_2\frak{B}_2} e\left( \beta F(\overrightarrow{r}) \right) d \overrightarrow{r}.
\end{gather*}

Let $\kappa(q)$ be the product of distinct prime factors of $(q,Q)$.
\begin{lem}\label{lem_Restricted sum}
Assume $|\beta| \leq n^{\nu-1}$, $1 \leq a \leq q \leq n^{\nu}$, $(a,q)=1$. Then
\begin{multline*}
    \sum_{\overrightarrow{z} \in \breve{\frak{B}}(n)} e\left( \left(\frac{a}{q} + \beta\right)F(\overrightarrow{z}) \right) \\ = \frac{|\breve{R}_Q|}{Q^{1+k_1+k_2}} \frac{\kappa(q)^{1+k_1+k_2}}{|\breve{R}_{\kappa(q)}|}
    \frac{1}{q^{1+k_1+k_2}} \breve{S}_{a,q} I(\beta) + O\left( n^{2 + 1/{k_0} - 1/k + 2\nu} \right)
\end{multline*}
\end{lem}

\begin{proof}
Let $\frak{q} = lcm(q,Q) = q \frac{Q}{\kappa(q)}$ and write $\overrightarrow{z} = \frak{q} \overrightarrow{a} + \overrightarrow{b}$, $0 \leq b_i < \frak{q}$. Let $\overline{A(\frak{q},\overrightarrow{b})} = \{ \overrightarrow{\eta} \in \mathbb{R}^{1+k_1+k_2} \; : \; \frak{q} \overrightarrow{\eta} + \overrightarrow{b} \in  X_0\frak{B}_0 \times X_1\frak{B}_1 \times X_2\frak{B}_2\}$ and $A(\frak{q},\overrightarrow{b}) = \overline{A(\frak{q},\overrightarrow{b})} \bigcap \mathbb{Z}^{1+k_1+k_2}$. One easily has
\begin{align*}
    &\sum_{\overrightarrow{z} \in \breve{\frak{B}}(n)} e\left( \left(\frac{a}{q} + \beta\right)F(\overrightarrow{z}) \right)\\
    &= \sum_{\overrightarrow{b} \in \breve{R}_{\frak{q}}} \sum_{ \overrightarrow{a} \in A(\frak{q},\overrightarrow{b}) } e\left( \frac{a}{q}F(\frak{q}\overrightarrow{a} + \overrightarrow{b}) \right) e\left( \beta F(\frak{q}\overrightarrow{a} + \overrightarrow{b}) \right) \\
    &= \sum_{\overrightarrow{b} \in \breve{R}_{\frak{q}}} e\left( \frac{a}{q}F(\overrightarrow{b}) \right) \sum_{ \overrightarrow{a} \in A(\frak{q},\overrightarrow{b}) } e\left( \beta F(\frak{q}\overrightarrow{a} + \overrightarrow{b}) \right).
\end{align*}

Let $h: \mathbb{R}^{1+k_1+k_2} \rightarrow \mathbb{C}$ be a map defined by $h(\overrightarrow{\eta}) = e(\beta F(\frak{q} \overrightarrow{\eta} + \overrightarrow{b}))$. If $\overrightarrow{a}$ is an integral vector near $\overrightarrow{\eta}$ so that $| \eta_i - a_i| \leq 1/2$ for all $0 \leq i \leq k_1 + k_2$, and $M$ is the supremum of the absolute value of directional derivatives of $h(\overrightarrow{\eta})$ for $\overrightarrow{\eta} \in \overline{A(\frak{q},\overrightarrow{b})}$, then clearly $|h(\overrightarrow{\eta}) - h(\overrightarrow{a})| \ll M$. Observe that any convex body $U \subset \overline{A(\frak{q},\overrightarrow{b})}$ can be divided into unit boxes together with at most $O\left( \frac{X_0}{\frak{q}} \left( \frac{X_1}{\frak{q}} \right)^{k_1} \left( \frac{X_2}{\frak{q}} \right)^{k_2} \left( \frac{\frak{q}}{X_0} + \frac{\frak{q}}{X_1} + \frac{\frak{q}}{X_2} \right)\right)$ possible broken boxes. Hence
\begin{multline*}
    \left| \int_{\overline{A(\frak{q},\overrightarrow{b})}} h(\overrightarrow{\eta}) d\overrightarrow{\eta}  - \sum_{ \overrightarrow{a} \in A(\frak{q},\overrightarrow{b}) } h(\overrightarrow{a}) \right| \\ \ll \left|\overline{A(\frak{q},\overrightarrow{b})}\right| M + \frac{X_0 X_1^{k_1} X_2^{k_2}}{\frak{q}^{k_1+k_2}}\left( \frac{1}{X_0} +  \frac{1}{X_1} + \frac{1}{X_2} \right) \sup_{\overrightarrow{\eta}}|h(\overrightarrow{\eta})|
\end{multline*}
 where $\left|\overline{A(\frak{q},\overrightarrow{b})}\right|$ is the Lebesgue measure of $\overline{A(\frak{q},\overrightarrow{b})}$. Here
\begin{gather*}
    \left|\overline{A(\frak{q},\overrightarrow{b})}\right| \ll \frac{X_0 X_1^{k_1} X_2^{k_2}}{\frak{q}^{1+k_1+k_2}}, \quad \sup_{\overrightarrow{\eta}}|h(\overrightarrow{\eta})| = 1,
\end{gather*}
and
\begin{equation*}
M \ll \sum_{i=0}^{k_1 + k_2} \left| \frac{\partial}{\partial \eta_i} h(\overrightarrow{\eta}) \right| \ll \frak{q} |\beta| \left( X_0^{k_0-1} + X_1^{k_1-1} + X_2^{k_2-1} \right) \ll \frak{q} |\beta| n^{1 - 1/k}
\end{equation*}
which gives
\begin{align*}
    \left| \int_{\overline{A(\frak{q},\overrightarrow{b})}} h(\overrightarrow{\eta}) d\overrightarrow{\eta}  - \sum_{ \overrightarrow{a} \in A(\frak{q},\overrightarrow{b}) } h(\overrightarrow{a}) \right|
    &\ll \frac{n^{2+1/{k_0}}}{\frak{q}^{k_1 + k_2}} n^{1 - 1/k} |\beta| + \frac{n^{2 + 1/{k_0} - 1/k}}{\frak{q}^{k_1 + k_2}} \\
    &\ll \frac{n^{2 + 1/{k_0} - 1/k + \nu}}{\frak{q}^{k_1 + k_2}}.
\end{align*}
Writing $\frak{q} \overrightarrow{\eta} + \overrightarrow{b} = \overrightarrow{r}$,
\begin{align*}
\int_{\overline{A(\frak{q},\overrightarrow{b})}} h(\overrightarrow{\eta})d\overrightarrow{\eta} &= \int_{\overline{A(\frak{q},\overrightarrow{b})}} e\left( \beta F(\frak{q} \overrightarrow{\eta} + \overrightarrow{b}) \right) d\overrightarrow{\eta}\\
 &= \int_{X_0\frak{B}_0 \times X_1\frak{B}_1 \times X_2\frak{B}_2 } e\left( \beta F(\overrightarrow{r}) \right) d \overrightarrow{r} \frac{1}{\frak{q}^{1+k_1+k_2}}\\
 &= \frac{1}{\frak{q}^{1+k_1+k_2}}I(\beta).
\end{align*}
Since $\frak{q} = q \frac{Q}{\kappa(q)}$, one has
\begin{equation*}
    \sum_{\overrightarrow{b} \in \breve{R}_{\frak{q}}} e\left( \frac{a}{q}F(\overrightarrow{b}) \right) = |\breve{R}_{Q/\kappa(q)}|\sum_{\overrightarrow{b} \in \breve{R}_{q}} e\left( \frac{a}{q}F(\overrightarrow{b}) \right) = |\breve{R}_{Q/\kappa(q)}| \breve{S}_{a,q}.
\end{equation*}

It follows that
\begin{align*}
    &\sum_{\overrightarrow{z} \in \breve{\frak{B}}(n)} e\left( \left(\frac{a}{q} + \beta\right)F(\overrightarrow{z}) \right)\\
    &= \sum_{\overrightarrow{b} \in \breve{R}_{\frak{q}}} e\left( \frac{a}{q}F(\overrightarrow{b}) \right) \left( \frac{1}{\frak{q}^{1+k_1+k_2}} I(\beta) + O\left( \frac{n^{2 + 1/{k_0} - 1/k + \nu}}{\frak{q}^{k_1 + k_2}} \right) \right)\\
    &= \frac{|\breve{R}_{Q/\kappa(q)}|}{(Q/\kappa(q))^{1+k_1+k_2}} \frac{\breve{S}_{a,q}}{q^{1+k_1+k_2}} I(\beta) + O\left( |\breve{R}_{\frak{q}}|\frac{n^{2 + 1/{k_0} - 1/k + \nu}}{\frak{q}^{k_1 + k_2}} \right)\\
    &= \frac{|\breve{R}_Q|}{Q^{1+k_1+k_2}} \frac{\kappa(q)^{1+k_1+k_2}}{|\breve{R}_{\kappa(q)}|} \frac{\breve{S}_{a,q}}{q^{1+k_1+k_2}} I(\beta) +  O\left( \frak{q} n^{2 + 1/{k_0} - 1/k + \nu} \right)
\end{align*}
and the lemma follows.
\qed\end{proof}

We introduce typical notations for the major arcs now. Compared to the classical ones, the singular series contains some additional terms in its summands as we use the restricted sum over $\overrightarrow{z} \in \breve{\frak{B}}(n)$.

\begin{gather*}
\breve{A}_n(q) = \sum_{a\text{ mod* }q} \frac{\breve{S}_{a,q}}{q^{1+k_1+k_2}}e\left( -\frac{a}{q}n \right), \quad  \breve{\frak{S}}(X,n) = \sum_{q \leq X} \frac{\kappa(q)^{1+k_1+k_2}}{|\breve{R}_{\kappa(q)}|} \breve{A}_n(q)\\
    \breve{\chi}_n(p) = \sum_{l=0}^{\infty} \frac{\kappa(p^l)^{1+k_1+k_2}}{|\breve{R}_{\kappa(p^l)}|} \breve{A}_n(p^l).\\
\end{gather*}

In particular $\breve{A}_n(1) = 1$. Possible issues on the convergence of $\breve{\chi}_n(p)$ will be clarified later. We also define the singular integral
\begin{equation*}
\frak{J}(c) = \int_{-c}^c \int_{\frak{B}} e(\gamma F(\overrightarrow{\zeta}))d \overrightarrow{\zeta} e(-\gamma) d \gamma.
\end{equation*}

\begin{thm}\label{thm_integration over major arcs}
\begin{multline}\label{eqn_integration over major arcs}
\int_{\frak{M}} \sum_{\overrightarrow{z} \in \breve{\frak{B}}(n)} e(\alpha F(\overrightarrow{z}))e(-n \alpha )d\alpha \\ = \frac{|\breve{R}_Q|}{Q^{1+k_1+k_2}} \breve{\frak{S}}(n^{\nu},n) \frak{J}(n^{\nu}) n^{1+1/{k_0}}+ O\left( n^{1 + 1/{k_0} - 1/k + 5\nu} \right).
\end{multline}
\end{thm}

\begin{proof}
Let $\frak{J}(c,n) = \int_{-c}^c I(\beta) e(-n \beta) d\beta$.
Note that there are $O(n^{2\nu})$ pairs $(q,a)$ in the major arcs and each interval $\frak{M}_{(q,a)}$ is of length $2 n^{\nu-1}$. From Lemma~\ref{lem_Restricted sum},

\begin{align*}
&\sum_{q \leq n^{\nu}} \sum_{a\text{ mod*}q} \int_{\frak{M}_{(q,a)}}  \sum_{\overrightarrow{z} \in \breve{\frak{B}}(n)} e(\alpha F(\overrightarrow{z}))  e(-n \alpha) d \alpha\\
&= \sum_{q \leq n^{\nu}} \sum_{a\text{ mod*}q} \frac{|\breve{R}_Q|}{Q^{1+k_1+k_2}} \frac{\kappa(q)^{1+k_1+k_2}}{|\breve{R}_{\kappa(q)}|} \frac{\breve{S}_{a,q}}{q^{1+k_1+k_2}}  e\left(-\frac{a}{q}n\right) \frak{J}(n^{\nu-1},n)\\
&\quad \quad + O\left( n^{2\nu} n^{\nu-1} n^{2 + 1/{k_0} - 1/k +2\nu} \right)\\
&=  \frac{|\breve{R}_Q|}{Q^{1+k_1+k_2}} \breve{\frak{S}}(n^{\nu},n) \frak{J}(n^{\nu-1},n) + O\left( n^{1 + 1/{k_0} - 1/k +5\nu} \right).
\end{align*}

Put $\overrightarrow{r} = (X_0 \zeta_0, X_1\zeta_1,\cdots,X_1\zeta_{k_1},X_2\zeta_{k_1+1},\cdots,X_2\zeta_{k_1+k_2})$ and $\gamma = n\beta$ so that $\beta F(\overrightarrow{r}) = \gamma F(\overrightarrow{\zeta})$. Then
\begin{align*}
    \frak{J}(n^{\nu-1},n) &= \int_{-n^{\nu-1}}^{n^{\nu-1}} \int_{X_0 \frak{B}_0 \times X_1 \frak{B}_1 \times X_2 \frak{B}_2} e(\beta F(\overrightarrow{r})) d \overrightarrow{r} e(-n \beta) d \beta\\
    &= \int_{-n^{\nu}}^{n^{\nu}} X_0 X_1^{k_1} X_2^{k_2} \int_{\frak{B}} e(\gamma F(\overrightarrow{\zeta})) d \overrightarrow{\zeta} e(-\gamma)\frac{d \gamma}{n}\\
    &= \frak{J}(n^{\nu})n^{1 + 1/{k_0}}
\end{align*}
and \eqref{eqn_integration over major arcs} follows.
\qed\end{proof}

\section{The singular series: Bad primes}\label{sec_singular series I}

\begin{lem}\label{lem_S_aq_S_br_S_arbq_qr}
If $(q,r) = (a,q) = (b,r)=1$ then $\breve{S}_{a,q} \breve{S}_{b,r} = \breve{S}_{ar+bq,qr}$.
\end{lem}
\begin{proof}
Write $\overrightarrow{z} = (z_0, \overrightarrow{x}, \overrightarrow{y})$ and $\overrightarrow{\xi} = (\xi_0, \overrightarrow{\alpha}, \overrightarrow{\beta})$.

\begin{align*}
    &\sum_{\overrightarrow{z} \in \breve{R}_q} e\left( \frac{a}{q}F(\overrightarrow{z}) \right) \sum_{\overrightarrow{\xi} \in \breve{R}_r} e\left( \frac{b}{r}F(\overrightarrow{\xi}) \right)\\
    &= \sum_{\overrightarrow{z} \in \breve{R}_q} \sum_{\overrightarrow{\xi} \in \breve{R}_r}  e\left( \frac{a}{q}F(\overrightarrow{z}) + \frac{b}{r}F(\overrightarrow{\xi}) \right)\\
    &= \sum_{z_0 = 1}^q \sum_{\xi_0 = 1}^r \sum_{\overrightarrow{x} \in \breve{R}_q^{K}} \sum_{\overrightarrow{\alpha} \in \breve{R}_r^{K}} \sum_{\overrightarrow{y} \in \breve{R}_q^{E}} \sum_{\overrightarrow{\beta} \in \breve{R}_r^{E}}e\left( \frac{a}{q}F(\overrightarrow{z}) + \frac{b}{r}F(\overrightarrow{\xi}) \right).
\end{align*}
Let $s_0 = r z_0 + q\xi_0$ so that $s_0$ takes every value modulo $qr$. Since $q \in O_K$ is a unit modulo $r O_K$, for any $\overrightarrow{\alpha} \in \breve{R}_r^K$ one has $N_K(q\overrightarrow{\alpha}) \not\equiv 0$ (mod $p$) for all $p \in \breve{P}$ that divides $r$, i.e., $q\overrightarrow{\alpha} \equiv \breve{R}_r^K$. Because the map $\overrightarrow{\alpha}\cdot\overrightarrow{\varphi}\mapsto q \overrightarrow{\alpha}\cdot\overrightarrow{\varphi}$ gives a bijection between systems of representatives of $O_K$ modulo $r O_K$, we obtain a bijection $q \breve{R}_r^K \rightarrow \breve{R}_r^K$. In the same manner, the set
\begin{equation*}
    \{ \overrightarrow{u} = r \overrightarrow{x} + q\overrightarrow{\alpha}\; : \; \overrightarrow{x} \in \breve{R}_q^K,\; \overrightarrow{\alpha} \in \breve{R}_r^K  \}
\end{equation*}
is bijective to $\breve{R}_{qr}^K$ and so is
\begin{equation*}
    \{ \overrightarrow{v} = r \overrightarrow{y} + q\overrightarrow{\beta}\; : \; \overrightarrow{y} \in \breve{R}_q^E,\; \overrightarrow{\beta} \in \breve{R}_r^E  \}
\end{equation*}
to $\breve{R}_{qr}^E$. Therefore
\begin{multline*}
    \sum_{z_0 = 1}^q \sum_{\xi_0 = 1}^r \sum_{\overrightarrow{x} \in \breve{R}_q^{K}} \sum_{\overrightarrow{\alpha} \in \breve{R}_r^{K}} \sum_{\overrightarrow{y} \in \breve{R}_q^{E}} \sum_{\overrightarrow{\beta} \in \breve{R}_r^{E}}e\left( \frac{a}{q}F(\overrightarrow{z}) + \frac{b}{r}F(\overrightarrow{\xi}) \right)\\
    = \sum_{z_0=1}^{q} \sum_{\xi_0=1}^{r}e\left( \frac{a}{q}s_0^{k_0} + \frac{b}{r}s_0^{k_0} \right) \sum_{\overrightarrow{x} \in \breve{R}_{q}^K} \sum_{\overrightarrow{\alpha} \in \breve{R}_{r}^K} e\left( \frac{a}{q}N_K(\overrightarrow{u}) + \frac{b}{r}N_K(\overrightarrow{u}) \right)\\
    \quad\quad \sum_{\overrightarrow{y} \in \breve{R}_{q}^E} \sum_{\overrightarrow{\beta} \in \breve{R}_{r}^E} e\left( \frac{a}{q}N_E(\overrightarrow{v}) + \frac{b}{r}N_E(\overrightarrow{v}) \right),\\
\end{multline*}
which is $\sum_{\overrightarrow{s}\in \breve{R}_{qr}} e\left( \frac{ar+bq}{qr}F(\overrightarrow{s}) \right) = \breve{S}_{ar+bq,qr}$.
\qed\end{proof}
An immediate implication of Lemma~\ref{lem_S_aq_S_br_S_arbq_qr} is the following.
\begin{cor}\label{cor_An_q_An_r__An_qr}
If $(q,r) = 1$, then $\breve{A}_n(q)\breve{A}_n(r) = \breve{A}_n(qr)$.
\end{cor}

Observe that we obviously have $\frac{\kappa(q)^{1+k_1+k_2}}{|\breve{R}_{\kappa(q)}|} \frac{\kappa(r)^{1+k_1+k_2}}{|\breve{R}_{\kappa(r)}|} = \frac{\kappa(qr)^{1+k_1+k_2}}{|\breve{R}_{\kappa(qr)}|}$ for $(q,r)=1$. Hence if we assume the convergence of $\lim_{t \rightarrow \infty} \breve{\frak{S}}(t,n) = \breve{\frak{S}}(\infty,n)$ which will be established later, we get
\begin{equation*}
    \breve{\frak{S}}(\infty,n) = \prod_{p \in \mathbb{P}} \breve{\chi}_n(p).
\end{equation*}

Let $\breve{M}_n(q)$ be the number of solutions to $F(\overrightarrow{z}) \equiv n$ (mod $q$) with $\overrightarrow{z} \in \breve{R}_q$.

\begin{lem}\label{lem_An_q_to_Mn_q}
For $L \geq 1$
\begin{equation*}
    \frac{\breve{M}_n(p^L)}{p^{(k_1 + k_2)L}} = \frac{|\breve{R}_p|}{p^{1+k_1+k_2}} + \breve{A}_n(p) + \breve{A}_n(p^2) + \cdots + \breve{A}_n(p^L).
\end{equation*}

\end{lem}
\begin{proof}
Let $q = p^L$. It is obvious that $\breve{R}_{p^m}$ can be constructed from $\breve{R}_p$ for any $m \geq 1$, namely
\begin{equation*}
    \breve{R}_{p^m} := \{ \overrightarrow{z}\text{ mod $p^m$}\; : \; \overrightarrow{z} \equiv \breve{R}_p \}
\end{equation*}
so that $\frac{|\breve{R}_{p^m}|}{p^{m(1+k_1+k_2)}} = \frac{|\breve{R}_{p}|}{p^{1+k_1+k_2}}$.
For $d = (a,q) = p^l$, $l \leq L$, write
\begin{equation*}
    \breve{S}_{\frac{a}{d}\cdot d, \frac{q}{d}\cdot d }  = \sum_{\overrightarrow{z} \in \breve{R}_q } e\left( \frac{a/d}{q/d}F(\overrightarrow{z}) \right) = \begin{cases}  d^{1+k_1+k_2}\breve{S}_{a/d,q/d} & \text{if $d \neq q$}\\
    |\breve{R}_q| & \text{ if $d=q$}.
    \end{cases}
\end{equation*}
It follows that
\begin{align*}
    \breve{M}_n(q) &= \sum_{\overrightarrow{z} \in \breve{R}_q} \frac{1}{q} \sum_{a=1}^q e\left( \frac{a}{q}(F(\overrightarrow{z})-n) \right)\\
    &= \frac{1}{q} \sum_{d|q} \sum_{\substack{a=1\\(a,q)=d}}^q \breve{S}_{\frac{a}{d}\cdot d, \frac{q}{d}\cdot d } e\left(  - \frac{a/d}{q/d}n \right)\\
    &= \frac{1}{q} \left( \sum_{l=0}^{L-1} \sum_{\substack{a=1\\(a,q)=p^l}}^{q} p^{l(1+k_1+k_2)} \breve{S}_{a/{p^l}, p^{L-l}} e\left( -\frac{a/{p^l}}{p^{L-l}}n \right) + |\breve{R}_q| \right)\\
    &= q^{k_1 + k_2} \left( \sum_{l=0}^{L-1} \sum_{\substack{a=1\\(a,q)=p^l}}^q \frac{\breve{S}_{a/{p^l}, p^{L-l}}}{(p^{L-l})^{1+k_1+k_2}}  e\left( -\frac{a/{p^l}}{p^{L-l}}n \right)   + \frac{|\breve{R}_q|}{q^{1+k_1+k_2}}  \right)\\
    &= q^{k_1 + k_2} \left( \sum_{l=1}^L \sum_{a\text{ mod*}p^l}\frac{\breve{S}_{a,p^l}}{p^{l(1+k_1+k_2)}} e\left( -\frac{a}{p^l}n \right) + \frac{|\breve{R}_q|}{q^{1+k_1+k_2}}  \right)\\
    &= q^{k_1 + k_2} \left( \frac{|\breve{R}_p|}{p^{1+k_1+k_2}} + \breve{A}_n(p) +\breve{A}_n(p^2) + \cdots + \breve{A}_n(p^L) \right).
\end{align*}
\qed\end{proof}

The estimation of $\breve{A}_n(p^L)$ for $p \in \breve{\mathbb{P}}$ uses next lemma.
\begin{lem}\label{lem_Hensel lifting}
Let $f(\overrightarrow{z})$ be an integral polynomial in $t+1$ variables for which $\overrightarrow{\zeta} = (\zeta_0, \zeta_1,\cdots,\zeta_{t})$ is a solution to $f(\overrightarrow{z}) \equiv n$ (mod $p^M$). Assume that there is an $i$ such that $u_i = v_p \left( \frac{\partial f(\overrightarrow{z})}{\partial z_i}|_{\overrightarrow{z} = \overrightarrow{\zeta}}\right) \leq \frac{M-1}{2}$. Let $N_i(M)$ be the number of solutions $\overrightarrow{\xi}$ modulo $p^M$ such that $\xi_j \equiv \zeta_j$ (mod $p^{2u_i+1}$) for $j \neq i$, $\xi_i \equiv \zeta_i$ (mod $p^{u_i+1}$) and $f(\overrightarrow{\xi}) \equiv n$ (mod $p^M$). Then $N_i(M) = p^{t(M-2u_i-1) + u_i}$.
\end{lem}

\begin{proof}
Without loss of generality, assume $i=0$ and let $\Delta_0 = \frac{\partial f(\overrightarrow{z})}{\partial z_0} \mid_{\overrightarrow{z} = \overrightarrow{\zeta}}$, $\Delta_0(a,b,c,\cdots) = \frac{\partial f(\overrightarrow{z})}{\partial z_0} \mid_{\overrightarrow{z} = (\zeta_0 + p^{u_0 + 1}a + p^{u_0 + 2}b + \cdots, \zeta_1,\cdots,\zeta_t)}$.  Note that
\begin{equation*}
    v_p(\Delta_0(a,b,c,\cdots)) = u_0
\end{equation*}
for all $a,b,c,\cdots \in \mathbb{Z}$. The Hensel's lemma can be applied to this situation in a form that starts from a higher power of $p$:
\begin{align*}
    &f(\zeta_0 + p^{u_0 + 1} z_0, \zeta_1,\cdots,\zeta_t) \equiv f(\overrightarrow{\zeta}) + \Delta_0 p^{u_0+1}z_0 \; \text{ mod \;$p^{2 u_0 + 2}$},\\
    &f(\zeta_0 + p^{u_0 + 1} z_0 + p^{u_0 + 2} z_0',\zeta_1,\cdots,\zeta_t)\\
    &\quad \equiv f(\zeta_0 + p^{u_0 + 1} z_0, \zeta_1,\cdots,\zeta_t) + \Delta_0(z_0) p^{u_0 + 2} z_0' \; \text{ mod \; $p^{2 u_0 + 3}$},
\end{align*}
and this process goes on. Thus for any $\xi_1,\cdots,\xi_t$ modulo $p^M$ with $\xi_j \equiv \zeta_j$ (mod $p^{2u_0+1}$), there exists a unique $r$ modulo $p^{M-2u_0-1}$ such that
\begin{equation*}
    f(\zeta_0 + p^{u_0 + 1} r, \xi_1,\cdots,\xi_t) \equiv n  \; \text{ mod \;$p^M$}.
\end{equation*}
Observe that this $r$ can be lifted to a number modulo $p^{M-u_0-1}$ in $p^{u_0}$ distinct ways. The value of $N_0(M)$ easily follows by counting the choices of $\xi_1,\cdots,\xi_t$ and $\zeta_0 + p^{u_0 + 1} r$ modulo $p^M$.
\qed\end{proof}

Let $\gamma= \gamma(p) = \min\{ v_p(k_1), v_p(k_2)\}$.
\begin{lem}\label{lem_An_p_Bad primes}
If $p \in \breve{\mathbb{P}}$ then $\breve{A}_n(p^L) = 0$ for all $L > 2\gamma+1$.
\end{lem}

\begin{proof}
Note that $p^{(k_1+k_2)L}\breve{A}(p^L) = \breve{M}_n(p^L) - p^{k_1+k_2}\breve{M}_n(p^{L-1})$ for $L \geq 2$. We show that all of the solutions that are counted in $\breve{M}_n(p^L)$ and $\breve{M}_n(p^{L-1})$ cancel out if $L \geq 2\gamma + 2$. Assume $\gamma = v_p(k_1)$. Suppose $\overrightarrow{z} = (z_0,\overrightarrow{x},\overrightarrow{y}) \in \breve{R}_{p^m}$ is a solution counted in $\breve{M}_n(p^m)$ where $p \in \breve{\mathbb{P}}$. Let $G = Gal(K|\mathbb{Q})$. Then $N_K(\overrightarrow{x}) = \prod_{\sigma \in G}\left( x_1 \varphi_1^{\sigma} + \cdots +  x_{k_1} \varphi_{k_1}^{\sigma}  \right)$, and we have
\begin{align*}
    \frac{\partial N_K(\overrightarrow{x})}{\partial x_i} &= \sum_{\sigma \in G} \varphi_i^{\sigma} \prod_{\substack{\tau \in G \\ \tau \neq \sigma}}\left( x_1 \varphi_1^{\tau} + \cdots +  x_{k_1} \varphi_{k_1}^{\tau}  \right)\\
    &= Tr_{K|\mathbb{Q}} \left( \varphi_i \prod_{\substack{\sigma \in G \\ \sigma \neq id}}\left( x_1 \varphi_1^{\sigma} + \cdots +  x_{k_1} \varphi_{k_1}^{\sigma}  \right) \right)\\
    &= N_K(\overrightarrow{x}) Tr_{K|\mathbb{Q}}\left( \frac{\varphi_i}{x_1 \varphi_1 + \cdots x_{k_1} \varphi_{k_1}} \right).
\end{align*}
Hence
\begin{equation*}
    x_1 \frac{\partial N_K(\overrightarrow{x})}{\partial x_1} + \cdots + x_{k_1} \frac{\partial N_K(\overrightarrow{x})}{\partial x_{k_1}} = N_K(\overrightarrow{x}) Tr_{K|\mathbb{Q}}(1) = {k_1} N_K(\overrightarrow{x}).
\end{equation*}
But $\overrightarrow{x} \in \breve{R}_{p^m}^K$, i.e., $N_K(\overrightarrow{x}) \not\equiv 0$ (mod $p$) which implies $x_i \frac{\partial N_K(\overrightarrow{x})}{\partial x_i} \not\equiv 0$ (mod $p^{\gamma+1}$) for some $i$. As described in the proof of Lemma~\ref{lem_Hensel lifting}, this solution $\overrightarrow{z}$ is one of the lifts of $\overrightarrow{\zeta} \in \breve{R}_{p^{2\gamma+1}}$ when $m \geq 2 \gamma + 1$ and so the solutions $\overrightarrow{z} \in \breve{R}_{p^{L}}$ and $\overrightarrow{z'} \in \breve{R}_{p^{L-1}}$ that are counted in $\breve{M}_n(p^L)$ and $\breve{M}_n(p^{L-1})$ are all lifts of the solutions $\overrightarrow{z} \in \breve{R}_{p^{2\gamma+1}}$. Lemma~\ref{lem_Hensel lifting} shows that the number of such lifts grows by a factor of $p^{k_1 + k_2}$ for each increment of $m$ in the modulus $p^m$ for $m \geq 2\gamma + 1$. Thus all of them are canceled in $\breve{M}_n(p^L) - p^{k_1+k_2}\breve{M}_n(p^{L-1})$.
\qed\end{proof}

The estimation of $\breve{A}_n(p^L)$ for $p \not\in \breve{\mathbb{P}}$ is a bit more complicated. Contrary to the classical problems in additive number theory for which the exponential sum $S_{a,p^L}$ splits into the product of many exponential sums over single variable, $S_{a,p^L}^K$ does not behave in the same way because $N_K(\overrightarrow{x})$ has many off-diagonal terms of the form $x_1^{a_1}x_2^{a_2}\cdots x_k^{a_k}$. Instead of obtaining a bound of exponential sums over a single variable, therefore, we focus on the properties of the norm map $N_{K|\mathbb{Q}}$ and it is here that the class field theory plays a role. A successful bound for $S_{a,p^L}^K$ comes in following sections.

\section{Algebraic preparation for the singular series}\label{sec_algebraic preparation for the singular series}

For the estimation of the exponential sum $S_{a,q}$, we translate the summands $\overrightarrow{x}$ in $S_{a,p^L}^K$ to a system of well-chosen representatives of the quotient ring $O_K/p^L O_K$.

\begin{lem}\label{lem_down to representatives}
Let $e,f,g$ be the ramification index, inertial degree and decomposition number of $p$ in $K|\mathbb{Q}$ and write $p O_K = \frak{p}_1^e \frak{p}_2^e \cdots \frak{p}_g^e$. Let
\begin{equation*}
    \alpha^{(m)} \in \left( \prod_{\substack{i = 1\\ i \neq m}}^g \frak{p}_i^{eL}\right) \backslash \frak{p}_m
\end{equation*}
and for each $m$ with $1\leq m \leq g$, let $\{ r_1^{(m)},r_2^{(m)},\cdots,r_{p^{efL}}^{(m)} \}$ be a system of representatives of $O_K$ modulo $\frak{p}_m^{eL}$. Then
\begin{equation*}
    S_{a,p^L}^K = \sum_{i_1 = 1}^{p^{efL}}\cdots\sum_{i_g = 1}^{p^{efL}} e\left( \frac{a}{q} N_{K|\mathbb{Q}}(\alpha^{(1)}r_{i_1}^{(1)} + \cdots + \alpha^{(g)}r_{i_g}^{(g)}) \right).
\end{equation*}
\end{lem}
\begin{proof}
Let $G = Gal(K|\mathbb{Q})$. We first recall that $N_K(\overrightarrow{x}) \in \mathbb{Z}[x_1,\cdots,x_{k_1}]$. Indeed, the coefficient of each term $x_1^{a_1}\cdots x_{k_1}^{a_{k_1}}$ in $N_K(\overrightarrow{x}) = \prod_{\sigma \in G}(x_1 \varphi_1^{\sigma} + \cdots + x_{k_1} \varphi_{k_1}^{\sigma})$ is an algebraic integer which is invariant under every $\sigma \in G$, so it is a rational integer. Hence for any integral vector $\overrightarrow{v}$ that is congruent to $\overrightarrow{x}$ modulo $p^L$ one has $N_K(\overrightarrow{v}) \equiv N_K(\overrightarrow{x})$ (mod $p^L$). Since $x_1 \varphi_1 + \cdots x_{k_1} \varphi_{k_1} \equiv v_1 \varphi_1 + \cdots v_{k_1} \varphi_{k_1}$ (mod $p^L O_K$) if and only if $x_i \equiv v_i$ (mod $p^L$) for all $i$, clearly
\begin{equation*}
    S_{a,p^L}^K = \sum_{\gamma \in \mathcal{R}} e\left( \frac{a}{p^L} N_{K|\mathbb{Q}}(\gamma) \right)
\end{equation*}
for any system $\mathcal{R}$ of representatives of $O_K/p^L O_K$.

We may work on a more general situation like the following. Let $I,J$ be integral ideals of $O_K$ that are relatively prime and $q,r$ be their absolute norms. Note that $q,r$ need not be relatively prime. Let $\{t_1,\cdots,t_q\}$ and $\{u_1,\cdots,u_r\}$ be systems of representatives of $O_K / I$ and $O_K/J$. Suppose we put $v_{i,j} = \beta t_i + \alpha u_j$ for some $\alpha \in I$, $\beta \in J$. Then $v_{i,j} \equiv \beta t_i$ mod $I$, and hence $v_{i,j} \equiv v_{i',j'}$ mod $I$ $\Leftrightarrow$ $\beta (t_i - t_{i'}) \in I$. If $\beta + I$ is not a zero divisor of $O_K/I$, this is equivalent to $t_i \equiv t_{i'}$ mod $I$. Thus $v_{i,j}$ for $1 \leq i \leq q$, $1 \leq j \leq r$ form a system of representatives of $O_K / IJ$ if and only if $\beta + I$ and $\alpha + J$ are not zero divisors in $O_K / I$ and $O_K/J$ respectively.

Because $I$ and $J$ are relatively prime, $I+J = O_K$ and there exist $\alpha \in I$ and $\beta \in J$ such that $\alpha \equiv 1$ mod $J$ and $\beta \equiv 1$ mod $I$. $1+I$ is clearly a unit in $O_K / I$ (in particular, not a zero divisor). The existence of $\alpha$, $\beta$ that satisfies the conditions mentioned above follows from this.

As an obvious generalization, assume $I^{(1)},\cdots,I^{(g)}$ are integral ideals that are relatively prime and let $\nu_m = N(I^{(m)})$. If $\{ r_1^{(m)},r_2^{(m)},\cdots,r_{\nu_m}^{(m)} \}$ is a system of representatives of $O_K/I^{(m)}$, there exist $\alpha^{(1)},\cdots,\alpha^{(g)}$ with $\alpha^{(m)} \in \prod_{\substack{i=1\\i\neq m}}^g I^{(i)}$ for which $\alpha^{(m)}+I^{(m)}$ is not a zero divisor in $O_K/I^{(m)}$ for all $m$. Writing $v_{i_1,\cdots,i_g} = \alpha^{(1)}r_{i_1}^{(1)} + \cdots + \alpha^{(g)}r_{i_g}^{(g)}$, the set
\begin{equation*}
\{ v_{i_1,\cdots,i_g}\;:\;1\leq i_1 \leq \nu_1,\; \cdots,\; 1\leq i_g \leq \nu_g \}
\end{equation*}
forms a system of representatives of $O_K/I^{(1)}\cdots I^{(g)}$. With the substitution $I^{(m)} = \frak{p}_m^{eL}$ and the choice of $\alpha^{(m)}$ as stated in the lemma, $\alpha^{(m)}$ is trivially not a zero divisor of $O_K/\frak{p}_m^{eL}$. This completes the proof.
\qed\end{proof}

The following is a well-known fact in algebraic number theory (for example, see Chapter 8 of ~\cite{Ribenboim}.)
\begin{prop}
Let $\frak{p}$ be a nonzero prime ideal of $O_K$ and $m \geq 1$. Let $\Gamma$ be a system of representatives of $O_K$ modulo $\frak{p}$ containing 0. Let $t \in \frak{p} \backslash \frak{p}^2$. Then $\Delta = \{ s_0 + s_1 t + \cdots + s_{m-1} t^{m-1}\;:\;s_i \in \Gamma \}$ is a system of representatives of $O_K$ modulo $\frak{p}^m$.
\end{prop}

Let $e$, $f$, $g$ be as in Lemma~\ref{lem_down to representatives}. We want to choose a system of representatives of $O_K$ modulo $\frak{p}_m^{eL}$ in a nice way. Consider the ideal class group $\mathcal{C}$ of $K$ and a class $[I] \in \mathcal{C}$ containing $I$. An analogue of Dirichlet's theorem on primes in arithmetic progression is that the prime ideals of $O_K$ are equi-distributed among ideal classes in $\mathcal{C}$ on a probabilistic point of view (Chapter 11 of ~\cite{Murty}). We simply take a weaker form of this for granted that every ideal class contains infinitely many prime ideals.

Firstly, consider $[\frak{p}_1] \in \mathcal{C}$. Then there is a prime ideal $\frak{a}_1$ such that $[\frak{a}_1] = [\frak{p}_1]^{-1}$ in $\mathcal{C}$ and $gcd(\frak{a}_1,\frak{p}_1\frak{p}_2\cdots\frak{p}_g) = O_K$. Now $\frak{p}_1 \frak{a}_1$ is principal, say, $t O_K$ and for $m \neq 1$
\begin{align*}
    \left( \left(\prod_{\substack{i=1\\i \neq m}}^g \frak{p}_i ^{eL}\right) \setminus \frak{p}_m \right) \bigcap t^{eL}O_K &= lcm\left( \prod_{\substack{i=1\\i\neq m}}^g \frak{p}_i ^{eL}, t^{eL}O_K \right) \setminus \frak{p}_m\\
    &= \left(\frak{a}_1^{eL} \prod_{\substack{i=1\\i \neq m}}^g \frak{p}_i ^{eL}\right) \setminus \frak{p}_m\\
    &\neq \emptyset
\end{align*}
whence one can choose $\alpha^{(m)} \in \left( \frak{a}_1^{eL} \prod_{\substack{i=1\\i \neq m}}^g \frak{p}_i ^{eL}\right) \setminus \frak{p}_m$, i.e., $\alpha^{(m)} \in \frak{a}_1^{eL} \frak{p}_1^{eL} = t^{eL}O_K$ and $\alpha^{(m)} / t^{eL} \in O_K$ for $m=2,3,\cdots,g$.

Similarly, we can choose $\frak{a}_1$, $\frak{a}_2$, $\ldots$, $\frak{a}_g$ so that $\frak{p}_i \frak{a}_i = t_i O_K$ is principal and $gcd(\frak{a}_i,\frak{p}_1\frak{p}_2\cdots\frak{p}_g) = O_K$ for all $i$. In particular $t_i \in \frak{p}_i \setminus \frak{p}_i^2$ and $v_p(N_{K|\mathbb{Q}}(t_i)) = f$. By the infinitude of the prime ideals in every ideal class, we can in addition assume  $(\frak{a}_i, \frak{a}_j) = O_K$ for $i \neq j$. It follows that there exists an element
\begin{equation*}
    \alpha^{(m)} \in \left( \prod_{\substack{i=1\\i \neq m}}^g \frak{a}_i ^{eL} \frak{p}_i^{eL} \right) \setminus \frak{p}_m
\end{equation*}
for each $m$ so that
\begin{equation*}
    \frac{\alpha^{(m)}}{\prod_{\substack{i=1\\i \neq m}}^g t_i^{eL}} \in O_K \setminus \frak{p}_m.
\end{equation*}

Let $\Gamma^{(1)},\cdots,\Gamma^{(g)}$ be systems of representatives of $O_K / \frak{p}_1, \cdots, O_K / \frak{p}_g$ that contain $0$. In the estimation of $S_{a,p^L}^K$ later, we choose the representatives that are of the form
\begin{equation}\label{eqn_representatives}
    \sum_{m=1}^g \alpha^{(m)} \left( s_0^{(m)} + s_1^{(m)} t_m + s_2^{(m)} t_m^2 + \cdots + s_{eL-1}^{(m)} t_m^{eL-1} \right)
\end{equation}
where $s_j^{(m)} \in \Gamma^{(m)}$.

\section{The singular series: Good primes}\label{sec_singular series II}

Let $\mathcal{R}$ be a system of representatives of $O_K/p^L O_K$ and $\mathcal{R}^{\ast} = \{ r \in \mathcal{R}\;:\; p \nmid N_{K|\mathbb{Q}}(r) \}$. Let $e$, $f$, $g$ be as in Lemma~\ref{lem_down to representatives}. Note that $\breve{S}_{a,p^L}^K = S_{a,p^L}^K$ if $p\not\in \breve{\mathbb{P}}$.
\begin{lem}\label{lem_Seperation of exponential sum}
Write $S_{a,p^L}^K = S_{a,p^L}^{K,1} + S_{a,p^L}^{K,2}$ where
\begin{gather*}
    S_{a,p^L}^{K,1} = \sum_{\gamma \in \mathcal{R}^{\ast}} e\left( \frac{a}{p^L} N_{K|\mathbb{Q}}(\gamma) \right), \quad S_{a,p^L}^{K,2} = \sum_{\gamma \in \mathcal{R} \setminus \mathcal{R}^{\ast}} e\left( \frac{a}{p^L} N_{K|\mathbb{Q}}(\gamma) \right).
\end{gather*}
\begin{enumerate}
  \item If $L \leq f$ then $S_{a,p^L}^{K,2} = \sum_{i=1}^g (-1)^{i-1} {{g}\choose{i}} p^{k_1 L-if}$
  \item If $L>1$ and $p \nmid k_1$ then $S_{a,p^L}^{K,1}=0$
  \item If $L=1$ and $p$ is unramified in $K|\mathbb{Q}$ then $S_{a,p^L}^{K,1}=-\frac{(p^f-1)^g}{p-1}$.
\end{enumerate}
\end{lem}

\begin{proof}
Assume $L \leq f$. Since $p \mid N_{K|\mathbb{Q}}(\gamma)$ if and only if $p^f \mid N_{K|\mathbb{Q}}(\gamma)$, in this case $S_{a,p^L}^{K,2}$ merely counts the number of nonunits of $O_K / p^L O_K$. Let $T$ be the number of nonunits in $O_K / p O_K$. The number of units in $O_K/p O_K$ is $\left( p^{ef} - p^{(e-1)f} \right)^g$, so
\begin{align*}
    T &= p^{k_1} - \left( p^{ef} - p^{(e-1)f} \right)^g = p^{k_1} - p^{efg} \left( 1-p^{-f} \right)^g\\
    &=\sum_{i=1}^g (-1)^{i-1} {{g}\choose{i}} p^{k_1-if}
\end{align*}
and $S_{a,p^L}^{K,2} = p^{(L-1)k_1}T = \sum_{i=1}^g (-1)^{i-1} {{g}\choose{i}} p^{k_1 L-if}$.

As for $S_{a,p^L}^{K,1}$, first assume $L > 1$ and $p \nmid k_1$. Let $\overrightarrow{x}$ be given by $\gamma = \overrightarrow{x}\cdot\overrightarrow{\varphi}$. Since $p \nmid k_1$, $p \nmid N_K(\overrightarrow{x})$ implies that there exists $i$ such that $\frac{\partial N_K(\overrightarrow{x})}{\partial x_i} \not\equiv 0$ (mod $p$) as shown in the proof of Lemma~\ref{lem_An_p_Bad primes}. If $m$ is any integer such that $m \equiv N_K(\overrightarrow{x})$ (mod $p$), Lemma~\ref{lem_Hensel lifting} shows that the number of $\overrightarrow{v}$ modulo $p^{L-1}$ satisfying $N_K(\overrightarrow{x} + p \overrightarrow{v}) \equiv m$ (mod $p^L$) is $p^{(k_1-1)(L-1)}$. Thus
\begin{equation*}
    S_{a,p^L}^{K,1} = p^{(k_1-1)(L-1)} \sum_{\substack{\overrightarrow{x}\text{ mod $p$}\\ p\;\nmid \; N_K(\overrightarrow{x})}} \sum_{z=1}^{p^{L-1}} e\left( \frac{a}{p^L}(N_K(\overrightarrow{x}) + pz) \right) = 0.
\end{equation*}

Now assume $L=1$ and $p$ is unramified in $K|\mathbb{Q}$ so that $N_{K|\mathbb{Q}}(r)$ takes every nonzero value modulo $p$. Considering $N_{K|\mathbb{Q}}(ru)$ for $u \in \mathcal{R}^{\ast}$, it is easy to see that the number of $r \in \mathcal{R}^{\ast}$ satisfying $N_{K|\mathbb{Q}}(r) \equiv m$ (mod $p$) is the same for each value of $m = 1,2,\cdots,p-1$. It follows that
\begin{align*}
    \sum_{r \in \mathcal{R}^{\ast}} e\left( \frac{a}{p} N_{K|\mathbb{Q}}(r) \right) &= \frac{(p^{ef}-p^{(e-1)f})^g}{p-1} \sum_{m=1}^{p-1} e\left( \frac{a}{p}m \right)\\
    &= - p^{k_1-fg} \frac{(p^f-1)^g}{p-1} = - \frac{(p^f-1)^g}{p-1}.
\end{align*}
\qed\end{proof}

We include a classical bound of an exponential sum for convenience. Let $S_{a,q}^0 = \sum_{m=1}^q e\left( \frac{a}{q}m^{k_0} \right)$.
\begin{lem}[Theorem 4.2 of \cite{Vaughan}]\label{lem_Vaughan_exponential sum}
For $(a,q)=1$, $S_{a,q}^0 \ll q^{1 - 1/{k_0}}$.
\end{lem}

Let $e$, $f$, $g$ be as in Lemma~\ref{lem_down to representatives}, $t_1$, $\cdots$, $t_g$ and $\alpha^{(1)}$, $\cdots$, $\alpha^{(g)}$ as described in Section~\ref{sec_algebraic preparation for the singular series}. Choose a system $\mathcal{R}$ of representatives of $O_K/p^L O_K$ whose elements are of the form given by \eqref{eqn_representatives}.

\begin{lem}\label{lem_S_a_q_estimation}
Let $p\not\in \breve{\mathbb{P}}$. For $L \geq 1$,
\begin{equation*}
    |S_{a,p^L}^{K}| \ll (p^L)^{k_1 - 1 + 2\log_p k_1} \quad \text{and} \quad |S_{a,p^L}^{E}| \ll (p^L)^{k_2 - 1 + 2\log_p k_2}.
\end{equation*}
\end{lem}
\begin{proof}
We prove the first inequality. Write $L = uf+v$, $1 \leq v \leq f$ and let $\mathcal{R}_i = \{ r \in \mathcal{R} \;:\; r \in \frak{p}_i \}$. The case $L=1$ easily follows from Lemmas ~\ref{lem_Seperation of exponential sum} and ~\ref{lem_Vaughan_exponential sum}, so we assume $L > 1$. Since $p \nmid k_1$, by Lemma~\ref{lem_Seperation of exponential sum} and the inclusion-exclusion principle
\begin{align*}
    S_{a,p^L}^{K} = S_{a,p^L}^{K,2} &= \sum_{r \in \bigcup_{m=1}^g \mathcal{R}_m} e\left( \frac{a}{p^{L-f}} \frac{N_{K|\mathbb{Q}}(r)}{p^f} \right)\\
    &= \sum_{l=1}^g (-1)^{l-1} \sum_{\substack{i_1,\cdots,i_l=1\\ i_1 < \cdots < i_l}}^g \sum_{r \in \bigcap_{m=1}^l \mathcal{R}_{i_m}} E(i_1,\cdots,i_l)\\
    &= \sum_1 + \sum_2
\end{align*}
where
\begin{equation*}
    E(i_1,\cdots,i_l) = e \left( \frac{a}{p^{L-lf}} \frac{N_{K|\mathbb{Q}}(t_{i_1}t_{i_2}\cdots t_{i_l})}{p^{lf}}   N_{K|\mathbb{Q}}\left(\frac{r}{t_{i_1}t_{i_2}\cdots t_{i_l}}\right)  \right),
\end{equation*}
and, in case $L \leq fg$, we wrote the sum over $l < \lfloor L/f \rfloor$ as $\sum_1$ and the sum over $\lfloor L/f \rfloor \leq l \leq g$ as $\sum_2$. Let $a_{(i_1,\cdots,i_l)} = a \frac{N_{K|\mathbb{Q}}(t_{i_1}t_{i_2}\cdots t_{i_l})}{p^{lf}}$. By the choice of $\mathcal{R}$, we have $(a_{(i_1,\cdots,i_l)},p)=1$ and $\frac{r}{t_{i_1}t_{i_2}\cdots t_{i_l}} \in O_K$ when $r \in \bigcap_{m=1}^l \mathcal{R}_{i_m}$. For $\sum_1$, observe that $\{ \frac{r}{t_{i_1}t_{i_2}\cdots t_{i_l}} \;:\; r \in \bigcap_{m=1}^l \mathcal{R}_{i_m}  \}$ runs through a system of representatives modulo $p^{L-lf} O_K$ by $p^{lf k_1 - lf}$ times. As for $\sum_2$, we first note that $E(i_1,\cdots,i_l) = 1$ always. Recall that an element of $\mathcal{R}$ is of the form
$$r = \sum_{m=1}^g \alpha^{(m)} \left( s_0^{(m)} + s_1^{(m)} t_m + s_2^{(m)} t_m^2 + \cdots + s_{L-1}^{(m)} t_m^{L-1} \right)
$$
where $s_j^{(m)} \in \Gamma^{(m)}$, $|\Gamma^{(m)}| = p^f$, and $r \in \mathcal{R}_i$ if and only if $s_0^{(i)} = 0$. The set $\bigcap_{m=1}^l \mathcal{R}_{i_m}$ consists of the numbers with $s_0^{(i_1)} = \cdots =s_0^{(i_l)}=0$ and hence contains $p^{lf(L-1)}p^{(g-l)fL} = p^{k_1L - lf}$ elements. We thus can write
\begin{align*}
    S_{a,p^L}^{K} &= \sum_{\substack{1 \leq l \leq g\\ l < \lfloor L/f \rfloor}} (-1)^{l-1} \sum_{\substack{i_1,\cdots,i_l=1\\ i_1 < \cdots < i_l}}^g p^{(k_1-1)lf} S_{a_{(i_1,\cdots,i_l)},p^{L-lf}}^K\\
    &\quad + \sum_{\lfloor L/f \rfloor \leq l \leq g} (-1)^{l-1} \sum_{\substack{i_1,\cdots,i_l=1\\ i_1 < \cdots < i_l}}^g p^{k_1L-lf}.
\end{align*}

Let $M(x)$ be the maximum value of $|S_{b,p^x}^K|$ among all $b \not\equiv 0$ (mod $p$) for $x \geq 1$, and $p^{k_1 x}$ when $x \leq 0$. Let $\theta_{(x)} = \theta_{K,p}(x)$ be the real number satisfying $M(x) = p^{x(k_1 - 1 + \theta_{(x)})}$. (In particular, $\theta_{(x)} = 1$ when $x < 0$ and by convention we put $\theta_{(0)} = 1$.) Then we have
\begin{align*}
    |S_{a,p^L}^K|   &\leq \sum_{l=1}^g {{g}\choose{l}} p^{lf(k_1-1)} M(L-lf)\\
    &\leq g \cdot \max_l \left\{ {{g}\choose{l}} p^{lf(k_1-1)} M(L-lf) \right\}\\
    &\leq g \cdot \max_l \left\{ {{g}\choose{l}} p^{lf(k_1-1) + (L-lf)(k_1-1+\theta_{(L-lf)})} \right\}\\
    &\leq \max_l \left\{ p^{(L-lf)(k_1-1+\theta_{(L-lf)}) + lf(k_1-1 + \frac{\log_p g}{f} + \frac{\log_p g}{lf})} \right\}
\end{align*}
whence for some $l$
\begin{equation*}
  L(k_1-1+\theta_{(L)}) \leq L(k_1-1+\theta_{(L-lf)}) + lf(-\theta_{(L-lf)} + \frac{l+1}{lf}\log_p g)
\end{equation*}
or
\begin{equation*}
  L \theta_{(L)} \leq \max_l \left\{ (L-lf)\theta_{(L-lf)} + (l+1)\log_p g \right\}.
\end{equation*}
From this expression, in an inductive way, one immediately has
\begin{equation}\label{eqn_L_theta_L}
(uf+v)\theta_{(uf+v)} \leq v \theta_{(v)} + 2 u \log_p g.
\end{equation}

Assume $v=1$ first. By Lemma~\ref{lem_Seperation of exponential sum}, $S_{a,p}^K = -\frac{(p^f-1)^g}{p-1} + \sum_{i=1}^g (-1)^{i-1} {{g}\choose{i}} p^{k_1-if}$. Here
\begin{align*}
    \frac{(p^f-1)^g}{p-1} &= \left( p^{f-1}+p^{f-2} + \cdots + 1 \right) \left( p^f -1 \right)^{g-1}\\
     &< 2 p^{f-1} p^{f(g-1)} = p^{k_1-1+\log_p 2},
\end{align*}
and observe that $0 < \sum_{i=1}^g (-1)^{i-1} {{g}\choose{i}} p^{k_1-if} \leq g p^{k_1-f}$. Thus
 \begin{align*}
    |S_{a,p}^K| < \max \left\{ p^{k_1-1+\log_p 2}, \; g p^{k_1-f} \right\} \leq k_1 p^{k_1-1} = p^{k_1-1+\log_p k_1}.
 \end{align*}

Now assume $v > 1$. Then
\begin{equation*}
    |S_{a,p^v}^K| = \left| \sum_{i=1}^g (-1)^{i-1} {{g}\choose{i}} p^{vk_1-if}\right| \leq k_1 p^{vk_1-f} = p^{v(k_1 - 1 + 1 - \frac{f}{v} + \frac{\log_p k_1}{v})}
\end{equation*}
and $1 - \frac{f}{v} + \frac{\log_p k_1}{v} \leq \frac{\log_p k_1}{f}$ where the equality holds when $v=f$. We have proved that $v \theta_{(v)} \leq \log_p k_1 $ in both cases; hence from \eqref{eqn_L_theta_L}
\begin{equation*}
    \theta_{(uf+v)} \leq \frac{ \log_p k_1 + 2 u \log_p g}{uf+v} \leq \frac{ (2u+1)\log_p k_1}{u+1} < 2 \log_p k_1.
\end{equation*}
\qed\end{proof}

The next one is an immediate corollary.
\begin{lem}\label{lem_An_p_Good primes}
If $p \not\in \breve{\mathbb{P}}$ then there exists an absolute constant $\delta > 0$ such that $\breve{A}_n(p^L) \ll \frac{1}{(p^L)^{1 + \delta}}$.
\end{lem}
\begin{proof}
Since $\breve{S}_{a,p^L} = S_{a,p^L}$ for $p \not\in \breve{\mathbb{P}}$, by Lemmas ~\ref{lem_Vaughan_exponential sum} and ~\ref{lem_S_a_q_estimation}
\begin{align*}
    \breve{A}_n(p^L) &\ll \frac{1}{(p^L)^{1+k_1+k_2}}\cdot p^L \cdot | S_{a,p^L}^0 S_{a,p^L}^K S_{a,p^L}^E |\\
    &\ll p^{ L\left( -k_1 - k_2 + (1 - \frac{1}{k_0}) + (k_1 - 1+ 2 \log_p k_1) + (k_2-1+2\log_p k_2) \right) }\\
    &< p^{ L\left(-1 - \frac{1}{k_0} + 2 \log_p k_1 k_2\right) }
\end{align*}
and $2 \log_p k_1 k_2 < \frac{1}{k_0}$ because $p \not\in \breve{\mathbb{P}}$.
\qed\end{proof}

Now the singular series is estimated.
\begin{thm}\label{thm_singular series}
There exist positive absolute constants $c_1$, $c_2$ that depend only on $K$ and $E$ such that
$c_1 < \breve{\frak{S}}(n^{\nu},n) < c_2$ for all sufficiently large $n$.
\end{thm}
\begin{proof}
Since $1 \leq \frac{\kappa(q)^{1+k_1+k_2}}{|\breve{R}_{\kappa(q)}|} \leq \frac{Q^{1+k_1+k_2}}{|\breve{R}_{Q}|} \ll 1$, the absolute convergence of
\begin{equation*}
    \lim_{t \rightarrow \infty} \breve{\frak{S}}(t,n)  = \breve{\frak{S}}(\infty,n)
\end{equation*}
follows from Corollary~\ref{cor_An_q_An_r__An_qr}, Lemmas ~\ref{lem_An_p_Bad primes} and \ref{lem_An_p_Good primes}. More precisely we have
\begin{equation*}
 | \breve{\frak{S}}(\infty,n) - \breve{\frak{S}}(n^{\nu},n)| \ll \frac{1}{n^{\nu\delta}}
\end{equation*}
so it suffices to show that $0 < c_1 < \breve{\frak{S}}(\infty,n) = \prod_{p \in \mathbb{P}} \breve{\chi}_n(p) < c_2$ for some constants $c_1$ and $c_2$.

For $p \not\in \breve{\mathbb{P}}$, $\frac{\kappa(p)^{1+k_1+k_2}}{|\breve{R}_{\kappa(p)}|} = 1$ so Lemma~\ref{lem_An_p_Good primes} gives
\begin{equation*}
    |\breve{\chi}_n(p) - 1| \leq \sum_{L=1}^{\infty} |\breve{A}_n(p^L)| \ll \sum_{L=1}^{\infty} \frac{1}{(p^L)^{1+\delta}} \ll \frac{1}{p^{1+\delta}}
\end{equation*}
which implies that there exists a prime $p_0$ depending only on $K$ and $E$ such that
\begin{equation*}
    \frac{1}{2} < \prod_{p \geq p_0} \breve{\chi}_n(p) < \frac{3}{2}.
\end{equation*}

Suppose $p < p_0$. Recall that every prime $p$ is unramified in at least one of $K|\mathbb{Q}$ and $E|\mathbb{Q}$, so assume $p$ is unramified in $K|\mathbb{Q}$. For $L \geq 2$, $\breve{M}_n(p^L) \geq N_n(p^L)$ where
\begin{equation*}
    N_n(p^L) = \left| \{ \overrightarrow{z} \in \breve{R}_{p^L}\;:\; F(\overrightarrow{z}) \equiv n\text{ mod $p^L$},\; N_K(\overrightarrow{x}) \not\equiv 0 \text{ mod $p$} \} \right|.
\end{equation*}
As in the proof of Lemma~\ref{lem_Seperation of exponential sum}, the number $s(m)$ of $\overrightarrow{x}$ modulo $p^L$ satisfying $N_K(\overrightarrow{x}) \equiv m$ (mod $p^L$) is the same for all $m \not\equiv 0$ (mod $p$). Since there are $(p^f-1)^g p^{(L-1)k_1}$ units in $O_K/p^L O_K = \frak{p}_1^{L}\cdots\frak{p}_g^{L}$, it is easy to see that $s(m) = \frac{(p^f-1)^g p^{(L-1)k_1}}{(p-1)p^{L-1}}$. Let $h(L)$ be the number of $(z_0,\overrightarrow{y})$ modulo $p^L$ such that $\overrightarrow{y} \in \breve{R}_{p^L}^E$ and $z_0^{k_0} + N_E(\overrightarrow{y})-n \not\equiv 0$ (mod $p$). For each $\overrightarrow{y} \in \breve{R}_{p^L}^E$, if $N_E(\overrightarrow{y}) - n \equiv 0$ (mod $p$) then there are $p^L - p^{L-1}$ $z_0$'s counted by $h(L)$. Otherwise there are at least $p^{L-1}$ $z_0$'s, so $h(L) \geq |\breve{R}_{p^L}^E| p^{L-1} = |\breve{R}_{p}^E| p^{(L-1)(k_2+1)}$. It follows that
\begin{multline*}
    N_n(p^L) \geq s(1)h(L) \geq \frac{(p^f-1)^g}{p-1} p^{(L-1)(k_1-1) + (L-1)(k_2+1)} |\breve{R}_p^E| \\ \geq (p-1)^{k_1-1} |\breve{R}_p^E| p^{(k_1+k_2)(L-1)}.
\end{multline*}
Let $e', f', g'$ be the ramification index, inertial degree and decomposition number of $p$ in $E|\mathbb{Q}$. Then
\begin{equation*}
    |\breve{R}_p^E| \geq |(O_E/pO_E)^{\ast}| = (p^{e'f'} - p^{(e'-1)f'})^{g'} = p^{k_2}\left(1 - \frac{1}{p^{f'}}\right)^{g'} \geq \left(\frac{p}{2}\right)^{k_2},
\end{equation*}
and hence
\begin{equation*}
    \frac{N_n(p^L)}{p^{(k_1+k_2)L}} \geq \frac{(p-1)^{k_1-1}}{2^{k_2}p^{k_1}}.
\end{equation*}
Now write $\breve{\chi}_n^{(L)}(p) = \sum_{l=0}^L \frac{\kappa(p^l)^{1+k_1+k_2}}{|\breve{R}_{\kappa(p^l)}|} \breve{A}_n(p^l)$. By Lemma~\ref{lem_An_q_to_Mn_q},
\begin{equation*}
    \breve{\chi}_n^{(L)}(p) = \frac{\kappa(p)^{1+k_1+k_2}}{|\breve{R}_{\kappa(p)}|} \frac{\breve{M}_n(p^L)}{p^{(k_1+k_2)L}} \geq \frac{\kappa(p)^{1+k_1+k_2}}{|\breve{R}_{\kappa(p)}|} \frac{(p-1)^{k_1-1}}{2^{k_2}p^{k_1}} =: u_p
\end{equation*}
and so $\breve{\chi}_n(p) \geq u_p > 0$.

As for the upper bound of $\breve{\chi}_n(p)$, if $p \in \breve{\mathbb{P}}$ then Lemmas~\ref{lem_An_p_Bad primes} and \ref{lem_An_q_to_Mn_q} give
\begin{multline*}
     \breve{\chi}_n(p) = \sum_{l=0}^{2 \gamma(p) + 1} \frac{\kappa(p^l)^{1+k_1+k_2}}{|\breve{R}_{\kappa(p^l)}|} \breve{A}_n(p^l) = \frac{p^{1+k_1+k_2}}{|\breve{R}_p|}  \frac{\breve{M}_n(p^{2 \gamma(p) + 1})}{p^{(k_1+k_2)(2 \gamma(p) + 1)}} \\ \leq \frac{p^{1+k_1+k_2}}{|\breve{R}_p|} \frac{|\breve{R}_p| p^{(1+k_1+k_2)2\gamma(p)}}{p^{(k_1+k_2)(2\gamma(p)+1)}} = p^{2\gamma(p)+1}.
\end{multline*}
If $p \not\in \breve{\mathbb{P}}$, $\breve{\chi}_n(p)$ converges absolutely by Lemma~\ref{lem_An_p_Good primes}. As the bound in Lemma~\ref{lem_An_p_Good primes} is independent of $n$, one can choose an upper bound $U_p$ of $\breve{\chi}_n(p)$ that depends only on $K$ and $E$. Therefore
\begin{equation*}
    c_1 = \frac{1}{2}\prod_{p < p_0} u_p < \breve{\frak{S}}(\infty,n) < \frac{3}{2} (k_1,k_2)^2 \prod_{p \in \breve{\mathbb{P}}} p \prod_{\substack{p < p_0 \\ p \not\in \breve{\mathbb{P}}}} U_p = c_2.
\end{equation*}

\qed\end{proof}

\section{The singular integral and minor arcs}\label{sec_singular integral}

The following proof is basically from \cite{Davenport}.
\begin{thm}\label{thm_singular integral}
We can choose $\frak{B}$ so that $\frak{J}(n^{\nu}) \rightarrow \frak{J}_0 > 0$ as $n \rightarrow \infty$.
\end{thm}
\begin{proof}
Choose small positive numbers $\phi_1$, $\cdots$, $\phi_{k_1+k_2}$ so that the real value $\phi_0$ that makes $\phi_0^{k_0} + N_K(\phi_1,\cdots,\phi_{k_1}) + N_E(\phi_{k_1+1},\cdots,\phi_{k_1+k_2}) = 1$ is positive, not equal to 1, and hence
\begin{equation*}
    \partial F / \partial \phi_0 \neq 0 \quad \text{and}\quad N_K(\phi_1,\cdots,\phi_{k_1}) + N_E(\phi_{k_1+1},\cdots,\phi_{k_1+k_2}) \neq 0.
\end{equation*}

Then $\overrightarrow{\phi} = (\phi_0,\cdots,\phi_{k_1+k_2})$ is a nonsingular solution to $F(\overrightarrow{\phi}) = 1$. Let $\frak{B}$ be a box centered at $\overrightarrow{\phi}$ with side lengthes $2 \lambda$. Write
\begin{align*}
    \frak{J}(\mu) &= \int_{-\mu}^{\mu} \int_{\frak{B}} e\left( \gamma F(\overrightarrow{\zeta}) \right) d\overrightarrow{\zeta} e(-\gamma) d\gamma\\
    &=\int_{\frak{B}} \frac{\sin 2\pi \mu(F(\overrightarrow{\zeta})-1)}{\pi (F(\overrightarrow{\zeta})-1)} d \overrightarrow{\zeta}\\
    &= \int_{-\lambda}^{\lambda}\cdots\int_{-\lambda}^{\lambda}  \frac{\sin 2\pi \mu (F(\overrightarrow{\phi} + \overrightarrow{\theta})-1)}{\pi (F(\overrightarrow{\phi}+\overrightarrow{\theta})-1)} d \overrightarrow{\theta}.
\end{align*}
Write $F(\overrightarrow{\phi}+\overrightarrow{\theta})-1 = c_0 \theta_0 + \cdots +c_{k_1+k_2} \theta_{k_1+k_2} + P_2(\overrightarrow{\theta}) + \cdots + P_k(\overrightarrow{\theta})$ where $P_m(\overrightarrow{\theta})$ is a homogeneous polynomial of degree $m$. Note that $c_0 = \partial F / \partial \theta_0 |_{\overrightarrow{\theta} = \overrightarrow{0}} = k_0 \phi_0^{k_0 -1} \neq 0$. Now for $r_0,r_1,r_2 \in \mathbb{R}$ consider the equation
\begin{equation*}
    r_0^{k_0} \phi_0^{k_0} + r_1^{k_1} N_K(\phi_1,\cdots,\phi_{k_1}) + r_2^{k_2} N_E(\phi_{k_1+1},\cdots,\phi_{k_1+k_2}) - 1 = 0.
\end{equation*}

Since both of $\phi_0$ and $N_K(\phi_1,\cdots,\phi_{k_1}) + N_E(\phi_{k_1+1},\cdots,\phi_{k_1+k_2})$ are nonzero, one can choose $r_0,r_1,r_2>0$ such that $k_0(r_0 \phi_0)^{k_0-1}=1$ and still satisfying
\begin{align*}
    &F(r_0\phi_0,r_1 \phi_1,\cdots,r_1\phi_{k_1},r_2\phi_{k_1+1},\cdots,r_2\phi_{k_1+k_2})-1 = 0 \quad \text{ and }\\ &N_K(r_1\phi_1,\cdots,r_1\phi_{k_1}) + N_E(r_2\phi_{k_1+1},\cdots,r_2\phi_{k_1+k_2}) \neq 0.
\end{align*}

 So we can assume $c_0 = k_0 \phi_0^{k_0-1} = 1$ from the beginning.

For $|\overrightarrow{\theta}| < \lambda$, we have $|F(\overrightarrow{\phi} + \overrightarrow{\theta})-1| < \sigma$ where $\sigma=\sigma(\lambda)$ is small for small $\lambda$. Put $F(\overrightarrow{\phi} + \overrightarrow{\theta})-1 = t$, and consider this as a map from $\theta_0$ to $t$. Then, if $\lambda$ is sufficiently small, the inverse function theorem tells us that $\theta_0$ can be expressed in terms of $t,\theta_1,\cdots,\theta_{k_1+k_2}$ as a power series
\begin{equation*}
    \theta_0 = t - c_1\theta_1 - \cdots - c_{k_1+k_2} \theta_{k_1+k_2} + \mathcal{P}(t,\theta_1,\cdots,\theta_{k_1+k_2})
\end{equation*}
where $\mathcal{P}$ is a multiple power series whose least degree terms are of degree at least 2. Hence $\partial \theta_0 / \partial t = 1 + \mathcal{P}_1(t,\theta_1,\cdots,\theta_{k_1+k_2})$ where $\mathcal{P}_1$ is a multiple power series without a constant term. By taking $\lambda$ sufficiently small, we can make $|\mathcal{P}_1(t,\theta_1,\cdots,\theta_{k_1+k_2})| < 1/2$ for $|\theta_1|,\cdots,|\theta_{k_1+k_2}| < \lambda$, $|t| < \sigma$. A change of variable from $\theta_0$ to $t$ gives
\begin{align*}
    \frak{J}(\mu) &= \int_{-\lambda}^{\lambda}\cdots\int_{-\lambda}^{\lambda}  \frac{\sin 2\pi \mu (F(\overrightarrow{\phi} + \overrightarrow{\theta})-1)}{\pi (F(\overrightarrow{\phi}+\overrightarrow{\theta})-1)} d\theta_1\cdots d\theta_{k_1+k_2} d \theta_0\\
    &\sim \int_{-\sigma}^{\sigma} \frac{\sin 2 \pi \mu t}{\pi t} V(t) dt
\end{align*}
where $V(t) = \int_{-\lambda}^{\lambda}\cdots\int_{-\lambda}^{\lambda} \left( 1 + \mathcal{P}_1(t,\theta_1,\cdots,\theta_{k_1+k_2}) \right) d \theta_1\cdots d\theta_{k_1+k_2}$ and we wrote $a\sim b$ to mean that the limit of their ratio equals $1$.

$V(t)$ is clearly a continuous function of $t$ for $|t|$ sufficiently small. We also observe that $V(t)$ has left and right derivatives at every value of $t$, and these derivatives are certainly bounded for $t$ in a small confined region. Therefore by Fourier integral theorem one has
\begin{equation*}
    \lim_{\mu \rightarrow \infty} \frak{J}(\mu) = \lim_{\mu \rightarrow \infty} \int_{-\sigma}^{\sigma} \frac{\sin 2 \pi \mu t}{\pi t} V(t) dt = V(0) =: \frak{J}_0.
\end{equation*}
But
\begin{align*}
    |V(0)| &= \left| \int_{-\lambda}^{\lambda}\cdots\int_{-\lambda}^{\lambda} \left( 1 + \mathcal{P}_1(0,\theta_1,\cdots,\theta_{k_1+k_2}) \right) d \theta_1\cdots d\theta_{k_1+k_2} \right|\\
    &>\int_{-\lambda}^{\lambda}\cdots \int_{-\lambda}^{\lambda} \frac{1}{2} d\theta_1\cdots d\theta_{k_1+k_2} > 0
\end{align*}
and the theorem follows.
\qed\end{proof}

We merely state the estimation on the minor arcs, which can be easily seen in ~\cite{Birch}.

\begin{thm}\label{thm_integration over minor arcs}
There exists $\delta > 0$ that depends only on $K$ and $E$ such that
\begin{equation*}
\int_{\frak{m}} \sum_{\overrightarrow{z} \in \breve{\frak{B}}(n)} e(\alpha F(\overrightarrow{z}))e(-n \alpha )d\alpha \ll n^{1+1/{k_0} -\delta}.
\end{equation*}
\end{thm}

\section{Conclusion}\label{sec_conclusion}

By Theorems~\ref{thm_integration over major arcs}, ~\ref{thm_singular series}, ~\ref{thm_singular integral} and ~\ref{thm_integration over minor arcs} one has
\begin{thm}\label{thm_conclusion}
The number of representations, $\breve{r}(n)$, of $n$ in the form $z_0^{k_0} + N_K(\overrightarrow{x}) + N_E(\overrightarrow{y})$ with $\overrightarrow{z} = (z_0,\overrightarrow{x},\overrightarrow{y}) \in \breve{\frak{B}}(n)$ satisfies
\begin{equation*}
    \breve{r}(n) = \frac{|\breve{R}_Q|}{Q^{1+k_1+k_2}} \frak{J}_0  \left(\prod_{p \in \mathbb{P}}\breve{\chi}_n(p)\right)  n^{1 + 1/{k_0}}+ o\left( n^{1 + 1/{k_0}} \right)
\end{equation*}
where $ 1 \ll \frak{J}_0  \prod_{p \in \mathbb{P}}\breve{\chi}_n(p) \ll 1$.
\end{thm}

A remark can be made on this result. Theorem~\ref{thm_conclusion} is optimal in the sense that the term $z_0^{k_0}$ is invincible to make the polynomial almost universal in any cases. We give here an example of a sum of two norms $N_K(\overrightarrow{x}) + N_E(\overrightarrow{y})$ which is locally universal but is not almost universal. This can be summarized as follows.

Let $f = f_{A,B}(\overrightarrow{u}) = N_{\mathbb{Q}(\sqrt{-A})|\mathbb{Q}}(u_1 + u_3\sqrt{-A}) + N_{\mathbb{Q}(\sqrt{-B})|\mathbb{Q}}(u_2 + u_4\sqrt{-B}) = u_1^2 + u_2^2 + A u_3^2 + B u_4^2$ be a quaternary quadratic form. Choose $A,B$ among prime numbers congruent to $1$ modulo $8$ that are sufficiently large so that $f$ cannot represent all of $\{$ $16$, $32$, $48$, $80$, $96$, $112$, $160$, $224$ $\}$. (In particular, $48 = 2^4\cdot 3 \neq u_1^2 + u_2^2$). $f$ is locally universal, since $N_{\mathbb{Q}(\sqrt{-A})|\mathbb{Q}}$ takes every unit value in $\mathbb{Z}_p$ for $p \neq 2,A$ and the same holds for $N_{\mathbb{Q}(\sqrt{-B})|\mathbb{Q}}$ in $\mathbb{Z}_p$ when $p\neq 2,B$. When $p=2$, we see that $A$ and $B$ are in $(\mathbb{Z}_2^{\ast})^2$ and $f$ is equivalent to $u_1^2 + u_2^2 + u_3^2 + u_4^2$ over $\mathbb{Z}_2$, which is universal by Lagrange theorem. But our choice of $A,B$ makes $f$ a 2-anisotropic form(\cite{Cassel}) with a large discriminant $d(f) = AB$. By the complete classification of almost universal quaternary quadratic forms(\cite{bochnak2007almost}), therefore, $f$ is not almost universal.


\bibliographystyle{amsplain}
\bibliography{AlmostUniversality}

\end{document}